\documentclass[12pt, a4paper]{amsart}

\usepackage{tikz}
\newcommand{\MyScale}{1} 

\address{ }
\usepackage[left=3cm,top=3.5cm,right=3cm]{geometry}
\usepackage{parskip, color}
\usepackage[T1]{fontenc}
\allowdisplaybreaks[2]
\usepackage{hyperref}
\usepackage{graphicx}
 \usepackage{xcolor}

\usepackage{marginnote}

\usepackage{verbatim} 

\DeclareMathOperator{\arcosh}{arcosh}
\newcommand{\pd}{\partial}

\newcommand{\R}{\mathbb{R}}

\newcommand{\x}{\lVert x \rVert}

\newcommand{\tn}{\textrm{tn}}

\newcommand\sbullet[1][.5]{\mathbin{\vcenter{\hbox{\scalebox{#1}{$\bullet$}}}}}
\usepackage{accents}
\newcommand*{\ddt}[1]{%
\accentset{\scriptsize\sbullet}{#1}}

\theoremstyle{remark}

\newtheorem{theorem}{Theorem}[section]
\newtheorem{lemma}[theorem]{Lemma}
\newtheorem{proposition}[theorem]{Proposition}
\newtheorem{corollary}[theorem]{Corollary}

\usepackage{fancyhdr}
\title{Explicit fundamental gap estimates for some convex domains in $\mathbb H^2$ }

\author{Theodora Bourni}
\address[Theodora Bourni]{University of Tennessee}
\email{\href{mailto: tbourni@utk.edu}{\nolinkurl{tbourni@utk.edu}}}

\author{Julie Clutterbuck}
\address[Julie Clutterbuck]{Monash University}
\email{\href{mailto:Julie.Clutterbuck@monash.edu}{\nolinkurl{Julie.Clutterbuck@monash.edu}}}

\author{Xuan Hien Nguyen}
\address[Xuan Hien Nguyen]{Iowa State University}
\email{\href{mailto: xhnguyen@iastate.edu}{\nolinkurl{xhnguyen@iastate.edu}}}

\author{Alina Stancu}
\address[Alina Stancu]{Concordia University}
\email{\href{mailto: alina.stancu@concordia.ca}{\nolinkurl{alina.stancu@concordia.ca}}}

\author{Guofang Wei}
\address[Guofang Wei]{UC  Santa Barbara}
\email{\href{mailto: wei@math.ucsb.edu}{\nolinkurl{wei@math.ucsb.edu}}}

\author{Valentina-Mira Wheeler}
\address[Valentina-Mira Wheeler]{University of Wollongong}
\email{\href{mailto: vwheeler@uow.edu.au}{\nolinkurl{vwheeler@uow.edu.au}}}

\thanks{
The research of Julie Clutterbuck was supported by grant FT1301013 of the Australian Research Council.
The research of Xuan Hien Nguyen was supported by grant 579756 of the Simons Foundation.
The research of Alina Stancu was supported by NSERC Discovery Grant  RGPIN 327635.
 The research of Guofang Wei was supported by NSF Grant DMS 1811558.
The research of Valentina-Mira Wheeler was supported by grant DP180100431 and DE190100379 of the Australian Research Council.}
\begin{document}
\maketitle

\begin{abstract}
	Motivated by an example of Shih \cite{shih1989counterexample}, we compute the fundamental gap of a family of convex domains in the hyperbolic plane $\mathbb H^2$,  showing that there are convex domains for which $\lambda_2 - \lambda_1 <  \frac{3\pi^2}{D^2}$, where $D$ is the diameter of the domain and $\lambda_1$, $\lambda_2$ are the first and second Dirichlet eigenvalues of the Laplace operator on the domain. The result contrasts with the case of domains in $\mathbb R^n $ or $\mathbb S^n$, where $\lambda_2 - \lambda_1 \geq \frac{3 \pi^2}{D^2}$\cite{fundamental,seto2019sharp,he2017fundamental,dai2018fundamental}.  We also show that the fundamental gap of the domains in Shih's example 
is still greater than $\tfrac 32 \frac{\pi^2}{D^2}$, even though the first eigenfunction of the Laplace operator is not log-concave.

\end{abstract}

\section{Introduction}
We consider the Laplace operator $- \Delta$ with Dirichlet boundary conditions on a compact domain $\Omega$ of $\mathbb H^2$. This operator has a discrete spectrum with $\infty$ as its accumulation point. If we list the sequence of eigenvalues in increasing order $\lambda_1 < \lambda_2 \leq \lambda_3 \leq \cdots$, the \emph{fundamental gap} is the difference between the first two eigenvalues
  \[
     \lambda_2 - \lambda_1>0.
  \]
  This spectral gap plays an important role in both mathematics and physics. For example, in quantum mechanics, it characterizes the energy difference between the ground state and the first excited state.

  Finding  a sharp lower bound for the fundamental gap of convex domains in $\R^n$ is a difficult problem with a long and rich history (see e.g. the recent survey article \cite{daifundamental}).
  One notable development was the estimate $\lambda_2 - \lambda_1 \geq \frac{\pi^2}{D^2}$ \cite{swyy-gap-estimate, yu-zhong-gap},
   where $D$ is the diameter of the domain, defined by
  \[
    D = \sup_{p, q \in \Omega} \| p-q\|. 
  \]  A key step in their proof was the fact that the first eigenfunction $u_1$ is log-concave (i.e. $\log u_1$ is concave), first proved by  Brascamp and Lieb \cite{MR0450480}.
  
  
  It was known that the estimate was not sharp:  the optimal gap was conjectured to be that obtained on an interval, with the saturated case happening as the domains degenerate to a one-dimensional strip. Finally, in 2011, the fundamental gap conjecture was resolved in  \cite{fundamental} by Andrews and Clutterbuck:  on a convex domain in $\mathbb R^n$ with Dirichlet boundary condition, $\lambda_2 - \lambda_1 \ge 3\pi^2/D^2$, where $D$ is the diameter of the domain. They used a new double-point technique in the proof. 
    
  Recently, Dai, He, Seto, Wang, and Wei (in various subsets) \cite{seto2019sharp,dai2018fundamental,he2017fundamental} generalized  the fundamental gap estimate to convex domains in $\mathbb S^n$, showing that $\lambda_2 - \lambda_1 \ge 3\pi^2/D^2$.     
    
  In both these settings, the log-concavity of the first eigenfunction plays an important role.   While mere log-concavity is sufficient to obtain the coarse estimate $\lambda_2-\lambda_1\ge \frac{\pi^2}{D^2}$, in order to obtain the optimal estimates in  \cite{fundamental, seto2019sharp,dai2018fundamental,he2017fundamental} it is shown that the first eigenfunction is \emph{super log-concave},  namely that the first eigenfunction is more log-concave than the first eigenfunction of the following one-dimensional model operator,
\begin{equation}
L_{n,K,D} (\phi) =  \phi''-(n-1)\tn_K(s)\phi'  \label{model}
\end{equation}
on $[-\frac{D}{2},\frac{D}{2}]$ with Dirichlet boundary condition. Here 
	\begin{equation*}
\tn_K(s) =
\begin{cases}
\sqrt{K}\tan(\sqrt{K}s), & K > 0 \\
0, & K=0 \\
-\sqrt{-K}\tanh(\sqrt{-K}s) & K <0
\end{cases}
\end{equation*}
where $K=0$ is the model for $\mathbb R^n$ and $K=1$ is the model for $\mathbb S^n$. 

Surprisingly, $K=-1$ is not a good model for $\mathbb H^n$. 
  Actually, the first eigenfunction of \eqref{model} when $K=-1$ is still log-concave. Indeed, from \cite[(2.16)]{seto2019sharp} we know that $(\log (\bar{\phi}_1))'' (0) = - \bar{\lambda}_1 < 0$ for all $K$, where $\bar{\phi}_1$ and $\bar{\lambda}_1$ are the (positive) first eigenfunction and the first eigenvalue of \eqref{model}. 
However, Shih  proved the existence of convex domains in $\mathbb H^2$ such that the first eigenfunction is not log-concave \cite{shih1989counterexample}. Therefore comparison to $\bar{\phi}_1$ with $K=-1$ will not work 
in the hyperbolic case. Very little is known for the fundamental gap lower bound estimate for $\mathbb H^n$ and in fact, one expects $3 \pi^2/D^2$ not to be a lower bound. 

In this paper, we estimate the fundamental gap and the diameter of a family of convex domains in the hyperbolic plane and confirm this intuition.  
  \begin{theorem} \label{main theorem}
	There are convex domains in  $\mathbb H^2$ such that 
	\begin{equation}
	\label{eq:SmallerFG}
	  \lambda_2 - \lambda_1 < 3\pi^2/D^2,
	\end{equation}
	where $D$ is the diameter of the domain. 
  \end{theorem}
  
  Our construction is motivated by Shih's example, and the above domains have very large diameter. For domains with small diameters, we conjecture that $3 \pi^2/D^2$ still works as a lower bound. In fact, we have a family of domains with gap greater than $3 \pi^2/D^2$ when the diameter is close to, but less than, $1$. See the beginning of Section \ref{gap-section}. 

In the domains for which \eqref{eq:SmallerFG} is true,  one can ask in addition whether $ \lambda_2 - \lambda_1 <\pi^2/D^2$ since Shih proved that the first eigenfunction is not log-concave for some of them. We show that the inequality does not hold and that the fundamental gaps of these examples are in fact greater than $3\pi^2/(2D^2)$, see the end of Section \ref{gap-section}. This illustrates that log-concavity of the first eigenfunction is not a necessary condition for the fundamental gap to be greater or equal to $\pi^2/D^2$. 

To the best of the authors' knowledge, the examples above give the first explicit fundamental gap estimates of the Dirichlet Laplacian for domains in the hyperbolic spaces in terms of the diameter. In \cite{Benguria-Linde2007}, an excellent upper bound for the Dirichlet gap was obtained in terms of the gap  of geodesic balls whose size was determined by the first eigenvalue of the domain in $\mathbb H^n$.

The organization of this paper is as follows:   In Section \ref{domain-section}, we set up the domain and describe how the eigenfunctions are found via  separation of variables, and identify the first two eigenvalues.  In Section \ref{first-eigenvalue-section}, 
we give some rough estimates for  the first two eigenvalues and the gap.  In Section \ref{diameter-section}, we estimate the diameter of the domains.   Finally, in Section \ref{gap-section}, we improve the estimate of the gap, thus proving Theorem \ref{main theorem}.

\subsection*{Acknowledgements} This research originated at  the  workshop ``Women in Geometry 2" at the  Casa Matem\'atica Oaxaca (CMO)  from June 23 to June 28, 2019.  We would like to thank CMO-BIRS for creating the opportunity to start work on this problem through their support of the workshop.

\section{The domains and their first two eigenvalues} \label{domain-section}

Let $\mathbb H^2$ be the hyperbolic space modelled by the Poincar\'e half-plane $\{ (x, y) \mid y >0\} = \{ (r, \theta) \mid r>0, \theta \in (0,\pi)\}$ with the metric
  \begin{equation}  \label{metric}
  g= ds^2 = \frac{ dx^2 + dy ^2}{y^2} = \frac{dr^2}{r^2 \sin^2 \theta} + \frac{d\theta^2}{\sin^2 \theta}.
  \end{equation}

In the orthonormal frame 
  \[
  e_1 = r \sin \theta \frac{\pd}{\pd r}, \quad e_2 = \sin \theta \frac{\pd}{\pd \theta},
  \]
	the non-vanishing Christoffel symbols are 
$	
	  \Gamma_{11}^2 = - \Gamma_{12}^1 = - \Gamma_{21}^1= \cos \theta.
	$
	With this information, it is straighforward to compute the covariant derivatives for any function $v$ and find 
  \begin{align}
  v_{11} &= r^2 \sin^2 \theta\ v_{rr} + r \sin^2 \theta\ v_r - \sin\theta \cos \theta\ v_{\theta},  \nonumber \\
  v_{22} & = \sin^2 \theta\ v_{\theta\theta} + \sin \theta \cos\theta\ v_{\theta}, \nonumber \\
  \Delta v & = v_{11} + v_{22} = r^2 \sin^2 \theta\ v_{rr} + \sin^2\theta\ v_{\theta\theta} + r \sin^2 \theta\ v_r.  \label{Laplace}
  \end{align}

  \subsection{The domains}
We consider the family of  domains 
  \[
  \Omega_{c,\theta_0, \theta_1} = \{(r,\theta) \mid 1 < r< e^{\pi/c},\  \theta_0< \theta <\theta_1\}, 
  \]
  where $c>0$, $\theta_0 \in (0, \tfrac{\pi}{2})$, and $\theta_1 \in (\tfrac{\pi}{2}, \pi)$ (see Figure \ref{fig:OmegaDomain}).  In these coordinates, geodesics are either vertical lines $x=c$ or half-circles  centred on the $x$-axis, so the sets $ \Omega_{c,\theta_0, \theta_1}$ are convex domains in $\mathbb H^2$.
	
	\begin{figure}[htbp]
    \begin{tikzpicture}[scale=\MyScale]
\draw[thick,->] (-4,0)--(4,0); 
\draw[thick,->] (0,0)--(0,5); 
\draw [blue, thick] ( {cos(45)}, {sin(45)})--({4.5*cos(45)},{4.5*sin(45)}); 
\draw [blue, thick] ( {cos(120)}, {sin(120)})--({4.5*cos(120)},{4.5*sin(120)}); 
\draw [blue,thick,domain=45:120] plot ({cos(\x)}, {sin(\x)});
\draw [blue,thick,domain=45:120] plot ({4.5*cos(\x)}, {4.5*sin(\x)});
\node[anchor=center] at (2.3,0.9) {\(\theta_0\)};
\node[anchor=center] at (3.15,0.9) {\(\theta_1\)};
\draw[densely dashed, thick, ->] (2.2,0) arc (0:45:2.2); 
\draw[densely dashed, thick, ->] (3,0) arc (0:120:3); 
\node[blue, anchor=center] at (1,3.5) {$\Omega_{c, \theta_0, \theta_1}$};
\node[anchor=west, blue] at ({.6*cos(45)}, {.6*sin(45)}) {$P$};
\node[anchor=west, blue] at ({5*cos(45)},{5*sin(45)}) {$Q$};
\node[anchor=east, blue] at ({5*cos(120)},{5*sin(120)}) {$R$};
\node[anchor=east, blue] at ({.6*cos(120)},{.6*sin(120)}) {$S$};
\node[anchor=west] at (0, 1.2) {$(0,1)$};
\node[anchor=west] at (0, 5) {$(0,e^{\pi/c})$};
    \end{tikzpicture}%
  \caption{Domain $\Omega_{c, \theta_0, \theta_1}= \{(r,\theta) \mid 1 < r< e^{\pi/c},\  \theta_0< \theta <\theta_1\} 
$.}
\label{fig:OmegaDomain}
\end{figure}
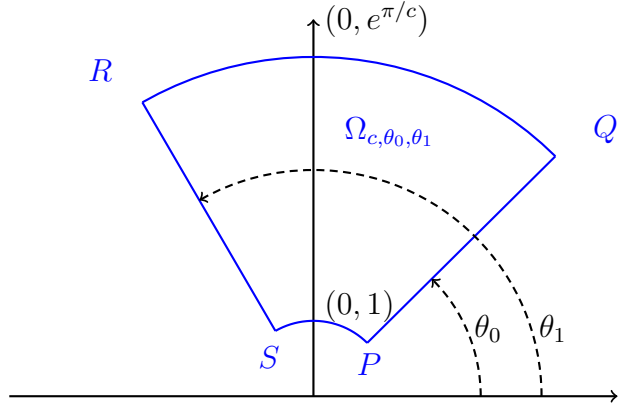

  \subsection{The separation of variables}
Because the metric $g$ from (\ref{metric}) is a warped product, we can use separation of variables for the eigenfunctions (see e.g. \cite[page 41]{Chavel}) and write
  \[
  u (r,\theta) = f(r)\, h (\theta). 
  \]

  We have, from our formulas for the Laplace operator (\ref{Laplace}), that $\Delta u = -\lambda u$ gives 
   	\begin{align*}
   	r^2 \sin^2 \theta \ f_{rr} \, h + r \sin^2 \theta \ f_r\, h + \sin^2 \theta \ f\, h_{\theta\theta} & = - \lambda f h,\\
   	\left(r^2 \frac{f_{rr}}{f} + r \frac{f_r}{f}\right) + \left(\frac{h_{\theta\theta}}{h}+ \lambda \csc^2 \theta \right)& =0.
   	\end{align*}   	
  We are looking for eigenfunctions with vanishing Dirichlet conditions on the boundary, hence we should solve the two eigenvalue equations
\begin{align}
\label{eq:f-equation}
r^2 f_{rr} + r f_r & = -\mu f, \ \ r \in [1, e^{\pi/c}],\\
\label{eq:h-equation}
h_{\theta\theta} + \lambda \csc^2 \theta \, h & =  \mu h, \ \ \theta \in [\theta_0, \theta_1], 
\end{align}
both with Dirichlet boundary conditions. With the change of variable $t = \log r$, equation (\ref{eq:f-equation}) becomes 
\begin{equation}
f_{tt} = -\mu f, \ \ t \in [0, \tfrac{\pi}{c}]. \label{eq:f-equation 2} \end{equation} In order for this to satisfy the boundary conditions, $\mu$ must be positive, so we set $\mu = (k c)^2, f(t) = \sin (k c t)$, where $k$ are nonzero integers. 

\subsection{The identification of the first two eigenvalues}
  The first Dirichlet  eigenvalue $\lambda_1$ of $\Delta u = -\lambda u$ on $\Omega_{c,\theta_0, \theta_1}$ 
corresponds to a strictly positive eigenfunction, which implies that $f$ in \eqref{eq:f-equation 2} is $\sin(ct)$, and so $\mu=c^2$.    Hence, $\lambda_1$ 
is given by the value $\lambda$ solving 
\begin{equation}  \begin{split}
h_{\theta\theta} + \lambda \csc^2 \theta \, h  =  c^2 h, \ \ \theta \in [\theta_0, \theta_1], \\
h(\theta_0)=h(\theta_1)=0,
\end{split}
 \label{h-first}
\end{equation}
for $h>0$.  We denote by  $\lambda^{c^2}_1$ the smallest $\lambda$ solving the above equation, so we have $\lambda_1=\lambda^{c^2}_1$. 

The second eigenvalue $\lambda_2$ corresponds to a sign-changing eigenfunction:  either $f$ or $h$ changes sign.   If $f$ changes sign, then $f$ in \eqref{eq:f-equation 2} is given by $\sin(2ct)$ and $\mu=4c^2$;   in this case $\lambda_2$ is given by $\lambda^{4c^2}_1$ solving 
\begin{equation}
h_{\theta\theta} + \lambda \csc^2 \theta \, h  =  4c^2 h, \ \ \theta \in [\theta_0, \theta_1],  \label{h-second}
\end{equation}
with $h>0$ and Dirichlet boundary conditions.   Otherwise $h$ changes sign, $f$ is positive and is given by $\sin(ct)$ with $\mu=c^2$;  then $\lambda_2$ is given by $\lambda^{c^2}_2$ solving \eqref{h-first} with $h$ changing sign exactly once.   

Thus, the second eigenvalue is   $\lambda_2=\min\lbrace\lambda^{4c^2}_1,\lambda^{c^2}_2\rbrace$.


\section{Estimates on the first and second eigenvalues} \label{first-eigenvalue-section}   \label{second-eigenvalue-section}
In this section, we give some rough estimates for  the first two eigenvalues and the fundamental gap. 

We define a convenient angle to simplify the exposition, so  let
    \label{def:ThetaAst}
    $$\theta_{\ast} := \min (\theta_0, \pi - \theta_1).$$
Note that $1 \le \csc^2 \theta \le\csc^2(\theta_{\ast})$, for all $\theta \in [\theta_0, \theta_1]$. 

We thus have the following estimate on the first eigenvalue of (\ref{eq:h-equation}).
\begin{lemma}  \label{first-eigenvalue-est}
  The first eigenvalue of (\ref{eq:h-equation}), denoted by $\lambda_1^\mu$, satisfies
	\begin{equation}
	  \label{eq:lambda1-estimate}
	  \sin^2(\theta_{\ast}) \left( \mu + \frac{\pi^2}{(\theta_1 - \theta_0)^2} \right) \le 	\lambda_1^\mu \le  \mu + \frac{\pi^2}{(\theta_1 - \theta_0)^2}. 
	\end{equation}
\end{lemma}
  \begin{proof}
Let $h$ be a solution of (\ref{eq:h-equation}). We multiply both sides of the equation by $h$, and integrate from $\theta_0$ to $\theta_1$, to obtain 
	\begin{eqnarray*}
	  \lambda & = & \frac{\int_{\theta_0}^{\theta_1} (|h_\theta|^2 + \mu {h}^2) d\theta}{\int_{\theta_0}^{\theta_1}(\csc^2 \theta) h^2 d\theta }  \\
	  & > &  \frac{1}{\csc^2(\theta_{\ast})} \left( \mu + \frac{\int_{\theta_0}^{\theta_1} |h_\theta|^2 d\theta}{\int_{\theta_0}^{\theta_1} h^2 d\theta} \right) \ge \sin^2(\theta_{\ast}) \left( \mu + \frac{\pi^2}{(\theta_1 - \theta_0)^2} \right),
\end{eqnarray*}
where in the last step we use Wirtinger's inequality $\int_0^D (h')^2\,dx\ge (\pi/D)^2 \int_0^D h^2\,dx$.  

To estimate the first eigenvalue from above, we choose the test function $\varphi = \sin \left( \tfrac{\theta -\theta_0}{\theta_1 -\theta_0} \pi \right)$ and recall that the first eigenvalue minimizes the Rayleigh quotient. Using $\csc^2 \theta \ge 1$,  we have the bound from above 
\[ \lambda_1^\mu \le 
 \mu + \frac{\pi^2}{(\theta_1 - \theta_0)^2}.
\qedhere
\] 	\end{proof}
	
	An alternate proof of  \eqref{eq:lambda1-estimate} using Sturm's comparison theorem can be found in the appendix.



\begin{lemma}\label{lemma on sign-changing h}
  We have the following estimate for $\lambda_2^{\mu}$, the second eigenvalue of \eqref{eq:h-equation}:
  \begin{equation}
    \label{eq:Second-evalue}
    \sin^2 \theta_{\ast}\left(\mu+ \frac{4\pi^2}{(\theta_1-\theta_0)^2}\right)\le \lambda_2^{\mu} \le \mu+ \frac{4\pi^2}{(\theta_1-\theta_0)^2}
  \end{equation}
\end{lemma}

\begin{proof}
  Let $h_2^\mu$ be an eigenfunction corresponding to the second eigenvalue $\lambda_{2}^{\mu}$ of (\ref{eq:h-equation}). Then there is a unique $\theta_2 \in (\theta_0, \theta_1)$ such that $h_2^\mu (\theta_2) = 0$. The eigenvalue $\lambda_2^\mu$ is the same as the first eigenvalue of $h_{\theta\theta} + \lambda \csc^2 \theta \, h  =  \mu h$ with zero Dirichlet boundary condition on either $[\theta_0, \theta_2]$ or $[\theta_2, \theta_1]$. Taking the interval with smaller length and applying Lemma~\ref{first-eigenvalue-est}, we get 
  \begin{equation*}
    \sin^2 \theta_{\ast} \left( \mu + \frac{4\pi^2}{(\theta_1 - \theta_0)^2} \right) \leq \lambda_2^{\mu}. 
  \end{equation*}
  For the upper bound, we apply Lemma~\ref{first-eigenvalue-est} with the longer interval. 
\end{proof}

Combining (\ref{eq:Second-evalue}) and Lemma~\ref{first-eigenvalue-est}, we have  $\lambda_2^{c^2} \ge \lambda_1^{4c^2}$ when 
  \begin{equation}
    \label{eq:Bound-c}
    \frac{\pi^2}{(\theta_1-\theta_0)^2} \frac{ (4 \sin^2 \theta_{\ast} -1)}{(4 - \sin^2 \theta_{\ast})} \geq c^2, \ \ \  \theta_{\ast} > \tfrac{\pi}{6}.
  \end{equation}
  
 Except Section~\ref{diameter-section},  in the rest of this article, we assume that $c>0$ satisfies \eqref{eq:Bound-c}, thereby the second Dirichlet eigenvalue of the Laplacian on $\Omega_{c, \theta_0, \theta_1}$ is $\lambda_1^{4c^2}$.  Geometrically, this corresponds to a domain, as shown in  Figure \ref{fig:OmegaDomain}, in which the opening angle is small in comparison to the vertical length.  
	
  \subsection{Rough estimate of the fundamental gap} 
  \begin{lemma}
    \label{lemma on second eigenvalue}
    Assume that $c,\ \theta_{\ast}$ satisfies \eqref{eq:Bound-c}.   Then the fundamental gap of $\Omega_{c,\theta_0, \theta_1}$ satisfies
  \begin{equation} 
    \label{eq:RoughFudamentalGap}
    3 \sin^2 \theta_{\ast} c^2< \lambda_2-\lambda_1 < 3c^2.
  \end{equation}
 Hence,  as $\theta_{\ast}$ approaches $\tfrac{\pi}{2}$, the gap approaches $3c^2$.
\end{lemma}

	\begin{proof}

	  Recall that the $\lambda_1$ is the first eigenvalue of \eqref{h-first} and, from our condition on $c$, $\lambda_2$ is the first eigenvalue of \eqref{h-second}. Let us denote by $h^{(1)}$ and $h^{(2)}$ the corresponding eigenfunctions, i.e.
	  \begin{align*}
	    h^{(1)}_{\theta\theta} + (\lambda_1 \csc^2 \theta -c^2) h^{(1)} &=0,\\
	    h^{(2)}_{\theta\theta} + (\lambda_2 \csc^2 \theta -4c^2) h^{(2)} &=0.
	  \end{align*}

	We argue by contradiction using Sturm comparison theorem I from the appendix. Suppose that $\lambda_2 \leq  \lambda_1+3 c^2 \sin^2 \theta_{\ast} $. We would have 
	  \[
	    \lambda_2 \csc^2 \theta - 4 c^2 
	    \leq \lambda_1 \csc^2 \theta + 3 c^2 \sin^2 \theta_{\ast} \csc^2 \theta - 4c^2 
	    \leq \lambda_1 \csc^2 \theta -c^2,
	\]
	where the last inequality is strict at interior points, and so there is no possibility of the left- and right- hand terms being equivalent.  This would mean that $h^{(2)}(\theta_1) >0$, which contradicts the Dirichlet boundary conditions. The other inequality is proved similarly using the fact that $\csc^2 \theta \geq 1$. 
	\end{proof}

\section{Estimating the diameter}  \label{diameter-section}

We start by recalling the  well known distance formula between two points in the hyperbolic plane
	\begin{align} \notag
    \text{dist}\left( (x_1,y_1),(x_2,y_2) \right) &= \arcosh \left(1 + \frac{(x_2 - x_1)^2 + (y_2-y_1)^2)}{2 y_1 y_2} \right) \\
    \label{eq:distance}
    	&= \arcosh \left( \frac{(x_1^2+y_1^2) + (x_2^2+y_2^2) - 2x_1x_2}{2 y_1y_2}\right).
      \end{align}

      The last form of the distance shows that for any $r$, the distance from a point $(r \cos \alpha, r \sin \alpha)$ to another point $(r \cos \beta, r \sin \beta)$ depends only on the angles $\alpha$ and $\beta$ and not on the radius $r$. 

We label the corners of our domain: given $\Omega_{c,\theta_0, \theta_1}$, we use cartesian coordinates and set $P = (\cos \theta_0, \sin \theta_0)$, $Q = e^{\pi/c} (\cos \theta_0, \sin \theta_0)$, $R=e^{\pi/c} (\cos \theta_1, \sin \theta_1)$, and $S = (\cos \theta_1, \sin \theta_1)$ (see Figure \ref{fig:OmegaDomain}).  This convex domain has a  piecewise smooth boundary.  The top and bottom boundary components are geodesics, while the lateral boundaries are not.
  \begin{proposition}
    \label{prop:Diameter}
    The diameter $D_{c,\theta_0, \theta_1}$ of the domain $\Omega_{c,\theta_0, \theta_1}$ is given by 
    \[
      D_{c,\theta_0, \theta_1} = \max \{\textrm{dist}(P,Q), \textrm{dist}(P, R),  \textrm{dist}(R, S)\}.
    \]
  \end{proposition}

  \begin{proof}
    We consider the closure $\overline \Omega :=\overline \Omega_{c,\theta_0, \theta_1}$ of our domain. Because $\overline \Omega$ is compact, the diameter is achieved, so we can choose points $V$ and $W$ such that $D_{c, \theta_0, \theta_1} = \textrm{dist}(V, W)$. We denote by $\gamma$ the geodesic segment between $V$ and $W$;   $\gamma$ is either a segment of a circle centered on the $x$-axis or a vertical line.

    First, we observe that neither $V$ nor $W$ is in the interior of $\Omega_{c, \theta_1}$, otherwise one would be able to prolong $\gamma$ and obtain a distance longer than the diameter.

    Next, we will show that neither $V$ nor $W$ can be in the interior of a boundary segment, in other words both $V$ and $W$ must be end points of boundary segments, which we also refer to as corners of the domain.

Suppose that one of the points $V$, $W$, say $V$, is in the interior of the top boundary segment $\overline{RQ}$.   
Since $V$ is not a corner point, there exists $T$, $T'$ points on the top boundary which is also a geodesic segment, such that $V$ is the midpoint between $T$ and $T'$.   Since $\mathbb{H}^2$ has negative curvature, we obtain the contradiction
\[D= \textrm{dist}(V,W)< \frac12\left(\textrm{dist}(T,W)+  \textrm{dist}(T',W)\right)\le D.\]
The same argument also shows that neither $V$ nor $W$ can belong to the interior of the lower boundary segment, $\overline{SP}$.

Suppose now that one of the points, say $V $, belongs to the interior of the lateral segment $\overline{RS}$. The closed geodesic ball of radius $D_{c, \theta_0, \theta_1}$ centered at $W$ contains $\overline \Omega$ and the boundaries of the ball and the domain touch at $V$. Since the boundary of $\overline \Omega$ is smooth at $V$, the tangent directions to the ball and the domain match. By Gauss' Lemma, the geodesic $\gamma$, which is a radius of the ball, is perpendicular to $\partial \overline \Omega$. The only geodesic starting at $V$ and perpendicular to $\overline{RS}$ is the arc of circle centered at the origin. If $V = (r \cos \theta_1, r \sin \theta_1)$, then $W=(r \cos \theta_0, r \sin \theta_0)$ with the same $r$, and dist$(V,W)$ = dist$(Q, R)$. 

%
    
    The discussion implies that
    \[
      D_{c,\theta_0, \theta_1} = \max \{\textrm{dist}(P,Q), \textrm{dist}(P, R), \textrm{dist}(Q,R), \textrm{dist}(R, S)\},
    \]
    since $\textrm{dist}(S, Q) = \textrm{dist}(R, P)$ and $\textrm{dist}(S, P) = \textrm{dist}(Q,R)$.  We finish the proof by computing all the distances using formula \eqref{eq:distance}, where $ \Psi(\alpha, \beta) = \frac{e^{2 \pi/c}+1 - 2 e^{\pi/c} \cos \alpha \cos \beta}{2 e^{\pi/c}\sin\alpha \sin \beta}$:     \begin{align*}
      \textrm{dist}(P,Q) & = \arcosh(\Psi(\theta_0, \theta_0))& 
      \textrm{dist}(P, R) & = \arcosh(\Psi(\theta_0, \theta_1))\\
    \textrm{dist}(Q,R) & =  \arcosh\left( \frac{1 -  \cos\theta_0 \cos \theta_1}{ \sin \theta_0\sin \theta_1} \right) & %
    \textrm{dist}(R, S) & = \arcosh(\Psi (\theta_1, \theta_1)).
    \end{align*}
    Note that 
    \[
      \cosh(\textrm{dist}(P,R))-\cosh( \textrm{dist}(Q, R)) = \left( e^{\pi/c} -1 \right)^2 /\left( 2 e^{\pi/c} \sin\theta_0 \sin\theta_1 \right) >0, 
    \]
    so $\textrm{dist}(P,R) > \textrm{dist}(Q, R) $. 
  \end{proof}

  It is worth mentioning that $\Psi(\theta_i, \theta_i) \leq \Psi(\theta_{\ast}, \theta_{\ast})$ for $i =0,1$, so $\textrm{dist}(P,Q)$ and $\textrm{dist}(R, S) \leq \arcosh \left( \Psi(\theta_{\ast}, \theta_{\ast}) \right)$. Note also that $\Psi(\theta_i, \theta_i) \leq \Psi(\theta_{\ast}, \pi-\theta_{\ast})$ for $i=1,2$, so the diameter of $\Omega_{c, \theta_{\ast}, \pi-\theta_{\ast}}$ is achieved by dist($P, R$). From these remarks, and Proposition \ref{prop:Diameter}, we get the following estimates for the diameter. 
   \begin{corollary} The following double inequality holds for the diameter $D_{c, \theta_0, \theta_1}$:
    \begin{equation}
      \label{eq:EstimateDiameter}
      \arcosh\left( \csc \theta_{\ast} \cosh(\pi/c) \right)\le D_{c, \theta_0, \theta_1} \leq \arcosh\left( \csc^2 \theta_{\ast} \cosh (\pi/c) + \cot^2 \theta_{\ast}\right).
    \end{equation}
  \end{corollary}

  \begin{proof}
    The right inequality holds because the domain $\Omega_{c, \theta_0, \theta_1} \subseteq \Omega_{c, \theta_{\ast}, \pi-\theta_{\ast}}$ and the right hand side is the diameter of $\Omega_{c, \theta_{\ast}, \pi-\theta_{\ast}}$. The left inequality is proved by noting that the diameter is greater than the distance from $P$ (or $R$) to the point $(0, e^{\pi/c})$.
   \end{proof}

We will now estimate the diameter more explicitly in terms of $c$ in order to compare it with the fundamental gap. 

  \begin{lemma} \label{diam-c}
The following limit holds: 	$\frac{\pi^2}{c^2 D^2_{c,\theta_0, \theta_1}} \to 1$ as $c \to 0$ or $\theta_{\ast} \to \tfrac{\pi}{2}$.
  	\end{lemma}
 \begin{proof} 	
    We use \eqref{eq:EstimateDiameter} and the formula $\arcosh (x) = \ln(x + \sqrt{x^2 -1})$ to find
  \begin{align*}
    \arcosh \left(\csc \theta_{\ast} \cosh \left( \pi/c \right)\right) & = \ln\left(\csc \theta_{\ast}\cosh\left( \pi/c \right) + \sqrt{\csc^2 \theta_{\ast} \cosh^2\left( \pi/c \right) -1}\right)\\
    & \geq \ln (\csc\theta_{\ast}) + \left( \pi/c \right).
    \end{align*}
    Writing $a = \cosh (\pi/c)+ \cos^2 \theta_{\ast}$ for brevity, we estimate the upper bound in a similar way 
    \begin{align*}
      \arcosh (a \csc^2 \theta_{\ast} )
      &= 2 \ln(\csc \theta_{\ast}) + \ln \left( a + \sqrt{a^2 - \sin^4 \theta_{\ast}}\right)\\
      & \le 2 \ln(\csc \theta_{\ast}) + \ln \left( a + \sqrt{a^2 - 1} + \sqrt{1 - \sin^4 \theta_{\ast}}\right).
    \end{align*}
    We remark that both $\ln$ and $\arcosh$ are concave functions, and that $f(x+b) \leq f(x) + f'(x)\, b$ for concave functions, thereby
    \begin{align*}
      \arcosh(a \csc^2 \theta_{\ast}) 
      & \le 2 \ln(\csc \theta_{\ast}) + (\pi/c) + \frac{\cos^2 \theta_{\ast}}{\sinh(\pi/c)} + \frac{\sqrt{1 - \sin^4 \theta_{\ast}}}{a+\sqrt{a^2 -1}}\\
      & = 2 \ln(\csc \theta_{\ast}) + (\pi/c) + \eta (\theta_{\ast}, c),
    \end{align*}
    where $\eta (\theta_{\ast}, c) = \tfrac{\cos^2 \theta_{\ast}}{\sinh(\pi/c)} + \tfrac{\sqrt{1 - \sin^4 \theta_{\ast}}}{a+\sqrt{a^2 -1}} $, which goes to zero as $c$ tends to zero or $\theta_{\ast} \to \tfrac{\pi}{2}$. 

    Therefore, we have 
    \begin{equation}
      \label{eq:diameter-estimate}
       c^2 \left(1  + \frac{c}{\pi} \left(2\ln (\csc\theta_{\ast})+\eta \right)\right)^{-2} \leq \frac{\pi^2}{D^2_{c, \theta_0,\theta_1}} \leq  c^2 \left( 1 + \frac{ c}{\pi}\ln \left( \csc\theta_{\ast} \right) \right)^{-2} <  c^2. 
    \end{equation}
    This shows that $\frac{\pi^2}{c^2 D^2_{c,\theta_0, \theta_1}} \to 1$ as $c \to 0$ or $\theta_{\ast} \to \tfrac{\pi}{2}$. 
    \end{proof}

\section{Estimating the fundamental gap} \label{gap-section}

From Lemmas~\ref{diam-c} and \ref{lemma on second eigenvalue}, we have that the gap of the domains $\Omega_{c, \theta_0, \theta_1}$ approaches $\tfrac{3\pi^2}{D^2}$ as $\theta_{\ast} \to \tfrac{\pi}{2}$.  In fact, combining \eqref{eq:RoughFudamentalGap} and \eqref{eq:diameter-estimate} gives
\begin{equation}
\sin^2 \theta_{\ast}  (1 + \frac{c}{\pi} \ln (\csc \theta_{\ast} ))^2  < (\lambda_2-\lambda_1) \cdot \tfrac{D^2_{c, \theta_0,\theta_1} }{3\pi^2}  < \left(1  + \frac{c}{\pi} \left(2\ln (\csc\theta_{\ast})+\eta \right)\right)^{2} . 
 \end{equation}

The left hand side is $\ge1$ when $c > \pi$ and $\sin \theta_{\ast} \ge \exp (\tfrac{\pi}{c}-1)$. For any fixed $c> \pi$, and for all $\theta_{\ast}$ sufficiently close to $\tfrac{\pi}{2}$,  we have that $c, \ \theta_{\ast}$ satisfy $\sin \theta_{\ast} \ge \exp (\tfrac{\pi}{c}-1)$ and condition \eqref{eq:Bound-c}. For these domains, the gap satisfies the inequality $\lambda_2-\lambda_1 > \tfrac{3\pi^2}{D^2}$. 
 
In order to prove Theorem~\ref{main theorem} though, we have to improve the upper bound estimate of the fundamental gap in \eqref{eq:RoughFudamentalGap}. We use the variation method as in \cite{Lavine-gap} and Sturm comparison for Jacobi equations to obtain the estimate.  

 Since  we  assume that $c, \theta_{\ast}$ satisfy  \eqref{eq:Bound-c},  the first and second Dirichlet eigenvalues of $-\Delta$ on $\Omega_{c, \theta_0, \theta_1}$ are given by the first eigenvalues of \eqref{h-first} and \eqref{h-second}, respectively.

Consider a family of problems generalizing \eqref{eq:h-equation}, indexed by a parameter $t$
  \begin{equation}
  h''+ v(t) h=\mu(t) h \text{ on $I=[\theta_0, \theta_1]$,}
  \label{evalue}
  \end{equation}
  with vanishing Dirichlet boundary conditions.   Here $h(\theta) = h^t(\theta)$  depends on $t$, and $v$ also depends on $t$, via setting $v(t)=\lambda(t)\csc^2 \theta$. Let $\lambda (t)$ be the first eigenvalue for each $t$, which is smooth in  $t$, and $h^t(\theta)$ are all  first eigenfunctions, so $h^t(\theta) >0$ on $(\theta_0, \theta_1)$.


 Denoting derivatives with respect to $t$ as $\ddt{h}$, we get
  \begin{equation}
    \ddt{h}''+ \ddt{v} h+ v\ddt{h}=\ddt{\mu} h+ \mu\ddt{h} \text{ on $I$}.
  \label{evalue t}
  \end{equation}
To relate changes in $\mu$ with changes in $v$,  we multiply \eqref{evalue t} by $h$, integrate over $I$, and use \eqref{evalue} to find
  \begin{equation*}
  \ddt{\mu} \int h^2 d\theta= \int \ddt{v} h^2 d\theta.
  \end{equation*}
Therefore, if we {set}  $\mu(t)= c^2+3c^2 t$, which determines $\lambda(t)$, we have 
  \[
    3c^2\int h^2 d\theta= \int \ddt{\lambda}(\csc^2 \theta) h^2 d\theta= \ddt{\lambda} \int (\csc^2 \theta) h^2 d\theta.
  \]
  If we rearrange this as 
  \[ \ddt{\lambda}= 3c^2 \frac{\int (h^t)^2 d\theta}{  \int (\csc^2  \theta)(h^t)^2d\theta },\]
  and integrate over $t$ from $0$ to $1$,  recalling that $\lambda(0)=\lambda_1$ and $\lambda(1)=\lambda_2$, we find
  \begin{equation}
  \label{eq:gap-estimate-improved}
  \lambda_2 -\lambda_1 \le 3 c^2 \max_{t\in[0,1]}\frac{\int (h^t)^2 d\theta}{  \int (\csc^2  \theta)(h^t)^2d\theta }.  
  \end{equation}

  \begin{proposition}
    \label{prop:quotient-estimate}
   \begin{equation}\label{concrete} \max_t\frac{\int (h^t)^2 d\theta}{  \int (\csc^2  \theta)(h^t)^2d\theta} \leq 1 -\delta,
   \end{equation}
for some $\delta = \delta (\theta_0, \theta_1)>0$, independent of $c$.   
  \end{proposition}

  Before we start the proof of Proposition \ref{prop:quotient-estimate}, we will compare $h^t$ to an explicit function. 
  \begin{lemma}
    Define $\sigma_1 =\frac{\pi^2\csc^2 \theta_{\ast}}{(\theta_1-\theta_0)^2} +(\csc^2 \theta_{\ast} -1)4c^2$ and $w_1(\theta)= \frac{1}{\sqrt{\sigma_1}}\sin\sqrt{\sigma_1}(\theta-\theta_0)$. 
    For $t \in [0,1]$, the solution $h^t$ to \eqref{evalue} with $\mu(t) = c^2 + 3c^2 t$ and $v(t) = \lambda(t) \csc^2 \theta$ and boundary condition $h(\theta_0) = 0$, $h'(\theta_0) =1$, $h(\theta_1) =0$ satisfies
    \[
    h^t(\theta) \le w_1(\theta), \quad \textrm{ for }\theta \in (\theta_0, \tilde \theta), \ t \in [0,1],
    \]
    where $\tilde\theta=\pi/\sqrt{\sigma_1}+\theta_0$.
  \end{lemma}

%

  \begin{proof}
    First, note that from \eqref{eq:lambda1-estimate}, we have
  \[
    v(t)-\mu(t) \le \sup_t \left(\lambda(t)\csc^2 \theta -\mu(t)\right) \le  \sup_t \left(\lambda(t) \csc^2 \theta_{\ast}  - \mu(t)\right) \le \sigma_1.
  \]
  Since $w_1(\theta)$ satisfies \[ w_1''+\sigma_1 w_1=0, \ \ w_1(\theta_0) =0, \ w_1'(\theta_0) =1,
  \]
  and $w_1>0$ on $(\theta_0, \tilde \theta)$, by Sturm comparison theorem II in the Appendix, we have $h^t(\theta) \geq w_1(\theta)$ for all $\theta \in (\theta_0, \tilde \theta)$ and $t \in [0,1]$. 
  \end{proof}

  Similarly, since 
  \begin{equation*}
    v(t) - \mu (t) \ge \inf_t(\lambda (t) \csc^2 \theta - \mu(t) \ge \sigma_2, 
  \end{equation*} 
  where $\sigma_2 = \sin^2\theta_{\ast}\frac{ \pi^2}{(\theta_1-\theta_0)^2} - \cos^2\theta_{\ast}4c^2$, and $h^t(\theta)>0$ on $(\theta_0, \theta_1)$, using Sturm comparison theorem II in the appendix again, we have,  for all $t \in [0,1]$, 
	\begin{equation}
	  \label{eq:vgreaterthanu}
	  w_2(\theta)\ge h^t(\theta) \textrm{ on } (\theta_0, \theta_1), 
	\end{equation} 
	where $w_2(t)$ is the solution of $$ w''+\sigma_2 w=0, \ \ w(\theta_0) =0, \ w'(\theta_0) =1. $$
	

  \begin{proof}[Proof of Proposition \ref{prop:quotient-estimate}]
    We choose an angle $\alpha$ so that 
    \[
    \theta_0 < \alpha < \min\left( \frac{\pi}{2}, \tilde \theta\right). 
    \]
   Then  
  \begin{align*}
    \int_{\theta_0}^{\theta_1} (\csc^2 \theta) h^2 d\theta&= \int_{\theta_0}^{\theta_1}  h^2 d\theta + \int_{\theta_0}^{\alpha} (\csc^2 \theta -1) h^2 d\theta + \int_{\alpha}^{\theta_1} (\csc^2 \theta -1 ) h^2 d\theta  \\
      & \ge \int_{\theta_0}^{\theta_1}  h^2 d\theta + \int_{\theta_0}^{\alpha} (\csc^2 \theta -1) w^2 d\theta \\
      & \ge \int_{\theta_0}^{\theta_1}  h^2 d\theta +(\csc^2 \alpha -1) \int_{\theta_0}^{\alpha}  w^2 d\theta \\
      & = \int_{\theta_0}^{\theta_1}  h^2 d\theta +(\csc^2 \alpha -1) b,
\end{align*}
where $b = \frac{1}{2 \sigma_1} \left(  (\alpha - \theta_0) - \frac{1}{2 \sqrt{\sigma_1}}\sin (2 \sqrt{\sigma_1} (\alpha - \theta_0))\right)$ is a positive constant which does not go to zero when $c$ tends to zero.

By \eqref{eq:vgreaterthanu},  $ \int_{\theta_0}^{\theta_1}  h^2 d\theta$ is bounded from above for all $t \in [0,1].$  
Therefore there is a $\delta>0$ for which 
   \[
     \max_t\frac{\int h^2 d\theta}{  \int (\csc^2  \theta)h^2d\theta} \leq \max_t\frac{\int h^2 d\theta}{\int h^2 d\theta + (\csc^2 \alpha -1) b} \leq  1 -\delta.\qedhere
   \]
\end{proof}
	

We are now ready to prove our main theorem.
\begin{proof}[Proof of Theorem ~\ref{main theorem}]
Combining the estimates  \eqref{eq:gap-estimate-improved} and \eqref{concrete}, with Lemma~\ref{diam-c}, we get that for all  $c$ sufficiently small, and $\theta_{\ast} > \tfrac{\pi}{6}$, the fundamental gap of  the domains $\Omega_{c, \theta_0, \theta_1}$  is less than $ 3 \pi^2/D^2$. 
\end{proof}

In Shih's example, $\theta_0 = \tfrac{\pi}{4}, \theta_1 \in (\tfrac{\pi}{2}, \tfrac{3\pi}{4})$, $c^2 < \tfrac{\pi}{5} \cot (\tfrac{19}{40}\pi) \left( 1+\tfrac{\pi}{40} \cot (\tfrac{19}{40}\pi) \right)^{-1}$. Hence $\theta_{\ast} = \tfrac{\pi}{4}$ and $c, \theta_{\ast}$ satisfy the condition \eqref{eq:Bound-c}.  Then, we can use \eqref{eq:RoughFudamentalGap} to get the gap estimate  $\lambda_2 - \lambda_1 \ge \tfrac{3}{2} c^2$. By \eqref{eq:diameter-estimate}, $c^2 > \tfrac{\pi^2}{D^2}$, therefore the gap is strictly greater than $ \tfrac{3}{2} \tfrac{\pi^2}{D^2}$, i.e. $\lambda_2 - \lambda_1> \tfrac{3\pi^2}{2D^2}$. 

\section{Appendix}

For convenience, we state here two versions of  Sturm comparison for Jacobi equations.

 Sturm comparison theorem I:  For $i=1,2$, let $f_i>0$ satisfy 
 $${f_i}'' + b_i f_i=0 \text{ on }(0,x_i),$$
 and $f_i(0)=0$, $f_i(x_i)=0$.     Suppose that $b_1\ge b_2$.   Then $x_1\le x_2$.
 If $x_1=x_2$, then $b_1 \equiv b_2$. 
 
  Sturm comparison theorem II: (see e.g. \cite[Page 238-239]{doCarmo})   For $i=1,2$, let $f_i$ satisfy 
  $${f_i}'' + b_i f_i=0 \text{ on }(0,l),$$
  and $f_i(0)=0$,  $f'_i(0)=1$.     Suppose that $b_1\ge b_2$ and $f_1 >0$ on $(0,l)$.   Then $f_1\le f_2$ on $(0,l)$.
  If $f_1=f_2$ at $t_1 \in (0,l)$, then $b_1 \equiv b_2$ on $(0,t_1)$.

We present below an alternative proof of Lemma \ref{first-eigenvalue-est}, using 
 Sturm comparison theorem I.
 
\begin{proof}[Alternative proof of Lemma  \ref{first-eigenvalue-est}]   Let $h>0$ satisfy \eqref{eq:h-equation}
 on $(\theta_0,\theta_1)$.   

We  consider two comparison functions, $h_1>0$ and $h_2>0$, satisfying
  \begin{equation*}
(h_i)_{\theta\theta} + b_i h_i =0  \text{ on } (\theta_0, a_i),
  \end{equation*}
  where $b_1= \lambda \csc^2 \theta_\ast - \delta$, $b_2 = \lambda - \delta$, with the boundary conditions 
  $h_i(\theta_0) = 0 \text{ and }   h_i(a_i) = 0.$

Note that 
$$b_2 \le  (\lambda\csc^2\theta-\delta)\le b_1,$$
so the Sturm comparison theorem I implies that
 \begin{equation} \label{eq: theta estimate} a_1<\theta_1<a_2.
 \end{equation}  

 However $h_1=\sin(\sqrt{\lambda \csc^2 \theta_\ast-\delta}(\theta-\theta_0))$ with $a_1= \theta_0+ \pi (\lambda \csc^2 \theta_\ast-\delta)^{-1/2},$ and $h_2 = \sin(\sqrt{\lambda - \delta} (\theta- \theta_0)$ with $a_2 = \theta_0 + \pi(\lambda -\delta)^{-1/2}$. So \eqref{eq: theta estimate} implies
  \[
    \sin^2 \theta_{\ast} \left(  \delta  + \frac{\pi^2}{\left(\theta_1-\theta_0\right)^2}\right)<\lambda<\delta  + \frac{\pi^2}{\left(\theta_1-\theta_0\right)^2}. \qedhere
  \]
\end{proof}

\bibliographystyle{plain}
\bibliography{references}

 \end{document}